\input{diagrams.tex}
\diagramstyle[scriptlabels,height=8mm,width=8mm]

\documentclass[a4paper,12pt]{amsart}
\usepackage[utf8]{inputenc}
\usepackage[english]{babel} 

\usepackage[T1]{fontenc}

\usepackage{amsthm}
\usepackage{amsmath}
\usepackage{amsfonts}
\usepackage{mathtools}
\usepackage[colorlinks]{hyperref}
\usepackage{enumerate}

\usepackage{tikz,pgf}
\usetikzlibrary{calc}

\def\bdi{\begin{diagram}}
\def\edi{\end{diagram}}

\newtheorem{thm}{Theorem}[section]
\newtheorem{cor}[thm]{Corollary}
\newtheorem{lem}[thm]{Lemma}
\newtheorem{prop}[thm]{Proposition}
\theoremstyle{definition}
\newtheorem{defi}[thm]{Definition}
\newtheorem{defis}[thm]{Definitions}
\newtheorem{conj}[thm]{Conjecture}
\newtheorem{conv}[thm]{Convention}
\newtheorem{nota}[thm]{Notation}
\newtheorem{rem}[thm]{Remark}
\newtheorem{rems}[thm]{Remarks}
\newtheorem{exa}[thm]{Example}
\newtheorem{exas}[thm]{Examples}
\newtheorem{prob}[thm]{Problem}
\newtheorem{probs}[thm]{Problems}
\newtheorem{ques}[thm]{Question}
\newtheorem{sett}[thm]{Setting}
\newtheorem{sit}[thm]{}

\newcommand{\Spec}{\operatorname{{\rm Spec}}}

\newcommand{\Sing}{\operatorname{{\rm Sing}}}

\newcommand{\Ker}{ \operatorname{{\rm Ker}}}

\newcommand{\Aut}{ \operatorname{{\rm Aut}}}
\newcommand{\End}{ \operatorname{{\rm End}}}

\newcommand{\PGL}{\operatorname{{\rm PGL}}}
\newcommand{\SL}{ \operatorname{{\rm SL}}}

\newcommand{\supp}{ \operatorname{{\rm supp}}}
\newcommand{\lie}{ \operatorname{{\rm lie}}}

\def\sl{\mathfrak{sl}}

\def\codim{\mathop{\rm codim}}

\def\Pic{\mathop{\rm Pic}}

\def\card{\mathop{\rm card}}

\renewcommand{\epsilon}{\varepsilon}

\def\and{\quad\mbox{and}\quad}

\newcommand{\F}{\ensuremath{\mathbb{F}}}

\newcommand{\G}{\ensuremath{\mathbb{G}}}

\newcommand{\A}{\ensuremath{\mathbb{A}}}

\newcommand{\kk}[1]{\bk^{[#1]}}

\newcommand{\hU}{{\hat U}}

\newcommand{\hZ}{{\hat Z}}

\newcommand{\hE}{{\hat E}}

\newcommand{\hO}{{\hat O}}

\newcommand{\hp}{{\hat p}}
\newcommand{\hs}{{\hat s}}

\newcommand{\tB}{{\tilde B}}
\newcommand{\tC}{{\tilde C}}

\newcommand{\tF}{{\tilde F}}

\newcommand{\tM}{{\tilde M}}

\newcommand{\tU}{{\tilde U}}

\newcommand{\tX}{{\tilde X}}
\newcommand{\tY}{{\tilde Y}}
\newcommand{\tZ}{{\tilde Z}}

\newcommand{\tE}{{\tilde E}}

\newcommand{\tO}{{\tilde O}}

\newcommand{\te}{{\tilde e}}
\newcommand{\tp}{{\tilde p}}
\newcommand{\ts}{{\tilde s}}

\newcommand{\tx}{{\tilde x}}

\newcommand{\hY}{{\hat Y}}

\newcommand{\bY}{{\bar Y}}

\def\fg{{\mathfrak g}}
\def\fh{{\mathfrak h}}

\def\fh{{\mathfrak h}}

\newcommand{\cF}{{\ensuremath{\mathcal{F}}}}

\newcommand{\cE}{{\ensuremath{\mathcal{E}}}}
\newcommand{\cA}{{\ensuremath{\mathcal{A}}}}

\newcommand{\cO}{{\ensuremath{\mathcal{O}}}}

\renewcommand{\rho}{\varrho}

\def\bals#1\eals{\begin{align*}#1\end{align*}}
\def\bal#1\eal{\begin{align}#1\end{align}}

\def\SL{\mathop{\rm SL}}

\def\kk{{\mathbb K}}
\def\A{{\mathbb A}}

\def\NN{{\mathbb N}}
\def\ZZ{{\mathbb Z}}
\def\QQ{{\mathbb Q}}

\def\PP{{\mathbb P}}

\renewcommand{\phi}{\varphi}

\newcommand{\bnum}{\begin{enumerate}}
\newcommand{\enum}{\end{enumerate}}

\newcommand{\brem}{\begin{rem}}
\newcommand{\brems}{\begin{rems}}
\newcommand{\erem}{\end{rem}}
\newcommand{\erems}{\end{rems}}
\newcommand{\bprob}{\begin{prob}}
\newcommand{\eprob}{\end{prob}}
\newcommand{\bprobs}{\begin{probs}}
\newcommand{\eprobs}{\end{probs}}
\newcommand{\bques}{\begin{ques}}
\newcommand{\eques}{\end{ques}}
\newcommand{\bexa}{\begin{exa}}
\newcommand{\bexas}{\begin{exas}}
\newcommand{\eexa}{\end{exa}}
\newcommand{\eexas}{\end{exas}}
\newcommand{\bdefi}{\begin{defi}}
\newcommand{\edefi}{\end{defi}}
\newcommand{\bdefis}{\begin{defis}}
\newcommand{\edefis}{\end{defis}}
\newcommand{\bcor}{\begin{cor}}
\newcommand{\ecor}{\end{cor}}
\newcommand{\blem}{\begin{lem}}
\newcommand{\elem}{\end{lem}}
\newcommand{\bconv}{\begin{conv}}
\newcommand{\econv}{\end{conv}}
\newcommand{\bconj}{\begin{conj}}
\newcommand{\econj}{\end{conj}}
\newcommand{\bprop}{\begin{prop}}
\newcommand{\eprop}{\end{prop}}
\newcommand{\bthm}{\begin{thm}}
\newcommand{\ethm}{\end{thm}}
\newcommand{\bnota}{\begin{nota}}
\newcommand{\enota}{\end{nota}}
\newcommand{\bsit}{\begin{sit}}
\newcommand{\esit}{\end{sit}}
\newcommand{\be}{\begin{equation}}
\newcommand{\ee}{\end{equation}}
\newcommand{\bproof}{\begin{proof}}
\newcommand{\eproof}{\end{proof}}
\newcommand{\bsett}{\begin{sett}}
\newcommand{\esett}{\end{sett}}
\def\ba{\begin{array}}
\def\ea{\end{array}}

\begin{document}
\title[Gromov ellipticity of cones]{Gromov ellipticity 
of cones over projective manifolds}

\author{S.~Kaliman and M.~Zaidenberg}
\address{Department of Mathematics,
University of Miami, Coral Gables, FL 33124, USA}
\email{kaliman@math.miami.edu}
\address{Univ. Grenoble Alpes, CNRS, IF, 38000 Grenoble, France}
\email{mikhail.zaidenberg@univ-grenoble-alpes.fr}


\date{}
\maketitle

\begin{abstract} We find classes of projective 
manifolds that are elliptic  in the sense of 
Gromov and such that the affine cones over these 
manifolds also are elliptic off their vertices. 
For example, the latter holds 
for any generalized flag manifold of dimension $n\ge 3$ 
successively blown up in 
a finite set of points and infinitesimally near points. This also holds
for any smooth projective rational surface. 
For the affine cones, the Gromov ellipticity is 
a much weaker property than the flexibility. 
Nonetheless, it still implies the infinite transitivity 
of the action of the endomorphism monoid on these cones. 
\end{abstract}

{\footnotesize \tableofcontents}

\section*{Introduction}
Let $\kk$ be an algebraically closed field  
of characteristic zero and $\A^n$ resp. $\PP^n$ 
be the affine resp. projective 
$n$-space over $\kk$. All varieties and vector bundles 
in this paper are algebraic. 
By \emph{spray} resp. \emph{ellipticity} 
we mean algebraic spray resp. algebraic ellipticity. 
Let $X$ be a smooth algebraic variety.
Recall that a \emph{spray of rank $r$} 
over $X$ is a triple $(E,p,s)$ where
$p\colon E\to X$ is  a vector bundle 
of rank $r$ with zero section $Z$ 
and $s\colon E\to X$ is a morphism  
such that $s|_Z=p|_Z$. A spray $(E,p,s)$ is 
\emph{dominating at $x\in X$} if the restriction 
$s|_{E_x}\colon E_x\to X$ 
to the fiber $E_x=p^{-1}(x)$ is dominant 
at the origin $0_x\in Z\cap E_x$ 
of the vector space $E_x$.
The constructive subset $s(E_x)\subset X$ 
is called the \emph{$s$-orbit of $x$}.  
The variety $X$ is \emph{elliptic} if it admits 
a spray $(E,p,s)$ which is dominating 
at each point $x\in X$. 
The ellipticity is equivalent to the weaker properties 
of local ellipticity and subellipticity, 
see \cite[Theorem 0.1]{KZ23}. 
One says that $X$ is \emph{locally elliptic} 
if for any $x\in X$ there is a local spray 
$(E_x,p_x,s_x)$ defined in 
a neighborhood of $x$ and  dominating at $x$ 
such that $s_x$ takes values in $X$. 
The variety $X$ is called 
\emph{subelliptic} if it
admits a family of sprays $(E_i,p_i,s_i)$ defined 
on $X$ which is dominating at each 
point $x\in X$, that is,
\[T_xX=\sum_{i=1}^n {\rm d}s_i (T_{0_{i,x}} E_{i,x}) 
\quad \forall x\in X.\]

The ellipticity implies many other useful properties; 
it is important in 
the Oka-Grauert theory, see, e.g., \cite{Gro89}, 
\cite{For17a} and  \cite{For23}. In particular, 
the semigroup of endomorphisms $\End(X)$ 
of an elliptic variety $X$ acts highly 
transitivity of $X$, see Appendix A in 
Section \ref{app:ht}.

In this paper we establish the ellipticity of 
 affine cones over certain elliptic varieties, 
see Theorems \ref{con.t1} 
and \ref{con.t2}. Summarizing, we prove the 
following theorem. Recall that 
a \emph{generalized affine cone} $\hat Y$ 
over a smooth projective variety $X$ defined 
by an ample $\QQ$-divisor $D$ on 
$X$ is the affine variety
\[\hat Y=\Spec\left(\bigoplus_{n=0}^\infty 
H^0\left(X,\cO_X(\lfloor nD\rfloor)\right)\right),\]
see, e.g., \cite[Sec. 1.15]{KPZ13}. 
In the case where  $D$ is a hyperplane section of 
$X\subset \PP^n$ the cone $\hat Y$ 
is the usual affine cone over $X$. 
\bthm\label{mthm} Let $X\subset\PP^n$ be a smooth 
projective variety, $\hat Y$ 
be a generalized affine cone over $X$, and 
$Y=\hat Y\setminus\{v\}$ where $v$ 
is the vertex of $\hat Y$. Then $Y$ 
is elliptic in the following cases:
\begin{itemize}\item $X$ is a rational surface;
\item $X$ is obtained via a sequence of successive 
blowups of points starting from 
a variety $X_0$ of dimension 
$n\ge 3$ which belongs to class $\cA_0$ 
and verifies condition $(**)$ of Proposition \ref{con.p1}. 
The latter holds for any generalized flag variety $X_0=G/P$ 
of dimension 
$n\ge 3$ where $G$ is a reductive group and $P$ 
is a parabolic subgroup of $G$. In particular, this holds for
$X_0=\PP^n$, $n\ge 3$, the Grassmannians, etc.
\end{itemize}
\ethm
Recall that a variety $X$ belongs to class $\cA_0$ if $X$ 
can be covered by 
open sets isomorphic to $\A^n$. 
A smooth projective surface 
belongs to class $\cA_0$ if and only if it is rational. 
Any variety of class $\cA_0$ is elliptic, see 
\cite[Proposition 6.4.5]{For17a} and \cite[Theorem 0.1]{KZ23}.

It is known that a variety of class $\cA_0$ 
blown up along a smooth closed subvariety is 
elliptic, see \cite{LT17} and \cite{KZ23}; cf. also \cite[Proposition 3.5E]{Gro89}. 
The same holds for a smooth locally stably flexible variety, 
see \cite{KKT18}. The following question arises.

\medskip

\noindent {\bf Question.} \emph{Is it true that the ellipticity 
of affine cones $Y$ over a smooth variety $X$ survives the blowup 
of a smooth closed subvariety of $X$?
Or, more specifically, whether the curve-orbit property $(*)$, 
its enhanced version $(**)$ (see Definition \ref{def:co} 
and Proposition \ref{con.p1}) 
and their suitable analogs survive such blowups?} 

\medskip

Generalizing an earlier result 
of Forstneri\^{c} \cite{For17b} 
valid for complete elliptic varieties,  
Kusakabe \cite{Kus22} established that any smooth elliptic variety 
admits a surjective morphism $\A^{\dim Y+1}\to Y$. 
This fact immediately implies the
following corollary. 
\bcor\label{cor:ht} Let $Y$ be a smooth elliptic variety.
Then the endomorphism monoid $\End(Y)$ acts highly 
transitively on $Y$, that is, 
for any $k\ge 1$ and any two 
corteges of distinct points $(y_1,\ldots,y_k)$ and 
$(y'_1,\ldots,y'_k)$ in $Y$ there exists 
$f\in\End(Y)$ such that $f(y_i)=y_i'$. 
\ecor
This corollary can be applied in particular 
to  generalized affine cones $Y$ as in Theorem \ref{mthm}. 
Let us compare Theorem \ref{mthm} with analogous
 facts concerning 
the flexibility of affine cones. 
Recall that a (quasi)affine variety $X$ of dimension 
$n\ge 2$ is said to be \emph{flexible} 
if for any smooth point $x\in X$ there exist $n$ 
locally nilpotent derivations of the algebra 
$\cO_X(X)$ whose vectors at $x$ span 
the tangent space $T_xV$, see  \cite{AFKKZ13}. 
Any smooth flexible variety is elliptic. 
The flexibility of $X$ implies that for any natural number 
$m$ the automorphism group $\Aut(X)$ 
acts $m$-transitively on the smooth locus of $X$, 
see \cite{AFKKZ13} in the affine case and \cite{FKZ16} 
in the quasiaffine case. 

Let $X=X_d$ be a del Pezzo surface of degree $d\ge 3$  
anticanonically embedded in $\PP^{d}$, and let 
$\hY\subset \A^{d+1}$ be the affine cone over $X$. 
It is known that for $d\ge 4$ the cone $\hY$ 
is flexible, see \cite{Per13}; cf.\ also \cite[Section 4]{MPS18} 
and the survey article \cite{CPPZ21} 
for other examples of this kind. 
By contrast, if $d\le 3$ and $X_d\hookrightarrow\PP^n$ 
is a pluri-anticanonical embedding, 
then the algebra $\cO(\hY)$ does not admit 
any nonzero locally nilpotent derivation, 
see \cite{CPW16} and also (for $d=1,2$) \cite{KPZ11}. 
In these cases $\hY$ is not flexible. 
However, by Theorem \ref{mthm}
for any del Pezzo surface $X=X_d$, $d\in\{1,\ldots,9\}$ 
and any ample polarization $D$ on $X$ 
the smooth locus $Y$ of the generalized affine cone $\hY$ 
over $X$ corresponding to $D$ is elliptic. 
As a corollary, the smooth quasiaffine threefold $Y$ 
admits a surjective 
morphism $\A^4\to Y$
and the endomorphism monoid $\End(Y)$ 
acts highly transitively on $Y$, see Section \ref{app:ht}. 

The structure of the paper is as follows. 
In Section \ref{sec:prelim} we recall some general 
facts about Gromov sprays. In particular, 
we need a version of the Gromov lemma 
on extension of sprays, 
see Proposition \ref{lem:ext} and Appendix B 
(Section \ref{app:el}) for a proof. 
Besides, we show that under certain conditions a
spray admits a pullback 
via a blowup, see Proposition \ref{cr.l2}. 
In Section \ref{sec:cones} we notice first that 
the ellipticity of a projective variety $X$ is necessary 
for the ellipticity of the affine cone over $X$, 
see Proposition \ref{prop:Gro}. We do not know 
whether this condition alone is sufficient. 
However, the ellipticity of $X$ together 
with an additional condition called the 
\emph{curve-orbit property}, 
see Definition \ref{def:co} and Proposition \ref{con.p1}, 
guarantees that 
the affine cone over $X$ is elliptic, see Corollary \ref{spr.cor2}. 
Theorem \ref{mthm} is proven in 
Section \ref{sec:A0} for varieties of dimension $\ge 3$ 
and in Section \ref{sec:surf} for surfaces, see Theorems 
\ref{con.t1} and \ref{con.t2}, 
respectively. Using results 
of Forstneri\^{c} \cite{For17b} and Kusakabe \cite{Kus22}, 
in  Appendix A (Section \ref{app:ht}) 
we deduce Corollary \ref{cor:ht}; see also Proposition 
\ref{prop:ht} and Corollary \ref{cor:Kus2} 
for  somewhat stronger results.  
\section{Preliminaries}\label{sec:prelim}
In this section $X$ is a smooth algebraic variety of positive dimension.

The following proposition is a version of the
Gromov Localization Lemma  adopted to our setting; 
see \cite[3.5.B]{Gro89} for a sketch of the proof, 
\cite[Propositions 6.4.1--6.4.2]{For17a} for a rigorous 
proof for quasi-projective varieties and \cite[Remark 3]{LT17}
for its extension 
to the general case. 
\begin{prop}\label{lem:ext} 
Let $D$ be a reduced effective divisor on $X$
and $(E,p,s)$  be a spray  
on $U=X\setminus {\rm supp}(D)$  with values in $X$ 
such that $p\colon E\to X$ is a trivial bundle of rank $r\ge 1$.
Then there exists a spray $(\tilde E,\tilde p,\tilde s)$ on $X$ 
whose restriction to $U$ is isomorphic to $(E,p,s)$ and 
such that 
$\tilde s|_{{\tilde p}^{-1}(X\setminus U)}=
\tilde p|_{{\tilde p}^{-1}(X\setminus U)}$ 
and for each $x \in U$ the $\ts$-orbit of 
$x$ coincides with the $s$-orbit of $x$. 
\end{prop}
For the reader's convenience 
we provide a proof in Appendix B, see Section \ref{app:el}.

Recall the notions of a locally nilpotent vector field and of its replica, 
see \cite{Fre06} and \cite{AFKKZ13}. 
\bdefi\label{ex:replica} Consider a regular vector field $\nu$ 
on a smooth affine variety 
$X=\Spec A$ and the associated derivation $\partial$ 
of the algebra $A=\cO_X(X)$. 
One says that $\nu$ and $\partial$ are \emph{locally nilpotent} 
if for any $a\in A$
one has $\partial^n(a)=0$ for some $n=n(a)\in\NN$.  Any 
$\G_{\mathrm a}$-action $\lambda$ on $X$ is the flow of a locally 
nilpotent vector field $\nu$. 
One writes in this case $\lambda(t)=\exp(t\nu)$. 
For any $h\in \Ker (\partial)$ 
the vector field $h\nu$ is also locally nilpotent. 
The $\G_{\mathrm a}$-action $\lambda_h(t):=\exp(th\nu)$ is called 
a \emph{replica} of $\lambda$. 
The points $x\in h^{-1}(0)$ are fixed by the replica $\lambda_h$. 
\edefi
In the sequel we deal with sprays generated 
by $\G_{\mathrm a}$-actions. 
\bdefi\label{ex:Ga-spray} Let $U$ be an affine 
dense open subset in $X$. 
Assume that $U$ is equipped with an effective 
$\G_{\mathrm a}$-action $\G_{\mathrm a}\times U\to U$. 
The latter morphism 
defines a rank $1$ spray $(L,p,s)$ on $U$ 
which we call 
a \emph{$\G_{\mathrm a}$-spray}, where  
$p\colon L\to U$ is a trivial line bundle. 
If $(L,p,s)$ is a $\G_{\mathrm a}$-spray on $U$ 
associated with a $\G_{\mathrm a}$-action 
$\lambda$ on $U$ and $\lambda_h$ is 
a  replica of $\lambda$, 
then the associated 
$\G_{\mathrm a}$-spray $(L,p,s_h)$ on $U$ 
also will be called  a \emph{replica} of $(L,p,s)$.
\edefi
\bdefi\label{def:ext-spray} 
Let $U=X\setminus\supp(D)$ where $D$ 
is an ample divisor on $X$ and let $(L,p,s)$ 
be a $\G_{\mathrm a}$-spray on $U$.
Due to Proposition \ref{lem:ext}  there exists 
a rank 1 spray $(\bar L,\bar p,\bar s)$ on $X$ whose restriction to $U$ 
is isomorphic to $(L,p,s)$ and such that 
$\bar s|_{{\bar p}^{-1}(X\setminus U)}=\bar p|_{{\bar p}^{-1}(X\setminus U)}$.
We call $(\bar L,\bar p,\bar s)$ an \emph{extended $\G_{\mathrm a}$-spray}. 
For a point $x\in U$ the $s$-orbit $O_x$ coincides with the $\G_{\mathrm a}$-orbit 
of $x$, while the $s$-orbit of $x\in X\setminus U$ is the singleton $\{x\}$. 
For any one-dimensional $\G_{\mathrm a}$-orbit $O$ in $U$ the restriction of 
$(\bar L,\bar p,\bar s)$ to $O$ is dominating. 
\edefi
It is known that under certain conditions dominating sprays 
can be lifted via blowups, see \cite[3.5E-E$''$]{Gro89}, \cite{LT17} 
and \cite{KKT18}. 
In the next proposition we provide a simple version of such results
adopted to our setup. 
\bprop\label{cr.l2}  Let  $(E,p,s)$ be an extended $\G_{\mathrm a}$-spray on $X$, 
$A\subset X$ 
be a smooth subvariety  of codimension at least $2$ such that the
$s$-orbit of any point $x\in A$ is a singleton, and 
$\sigma \colon \tX \to X$ 
be the blowup of $A$ in $X$. 
Then $(E,p,s)$ can be lifted to an extended $\G_{\mathrm a}$-spray 
$(\tE,\tp,\ts)$ on $\tX$ where 
$\tE=\sigma^*E$. Letting $\tilde\sigma\colon \tE\to E$ 
be the induced homomorphism, $\ts\colon \tE\to \tX$ 
satisfies $\sigma\circ\ts=s\circ\tilde\sigma$.
In particular, $\sigma$ sends 
the $\ts$-orbits to $s$-orbits.  Furthermore, if $(E,p,s)$ 
is dominating when restricted 
to a one-dimensional $s$-orbit $O$ and $\tO$ is 
a $\ts$-orbit in $\tX$
with $\sigma(\tO)=O$ then  $(\tE,\tp,\ts)$ is dominating  
when restricted to $\tO$. 
\eprop
\bproof 
Let $U\subset X$ be the affine dense open subset carrying 
an effective 
$\G_{\mathrm a}$-action $\lambda$ such that $(E,p,s)$ is extended 
from the associate 
$\G_{\mathrm a}$-spray on $U$, see Definition \ref{ex:Ga-spray}. 
Since $U$ is affine, 
$D=X\setminus U$ is a divisor. 
The $s$-orbit of $x\in U$ 
is a singleton if and only if $x$ is a fixed point of $\lambda$. 
By our assumption, each $x\in A\cap U$ is a fixed point of 
$\lambda$. So,
the ideal $I(A\cap U)$ is stable under $\lambda$.
By the universal property of blowups, 
see \cite[Corollary II.7.15]{Har04},
$\lambda$ 
admits a lift to a $\G_{\mathrm a}$-action $\tilde\lambda$ on 
$\tilde U=\sigma^{-1}(U)$ 
making the morphism 
$\sigma|_{\tilde U}\colon\tilde U\to U$ equivariant. 
The associated $\G_{\mathrm a}$-spray $\ts$ on $\tilde U$ satisfies 
$\sigma\circ\ts|_{\tU}=s\circ\tilde\sigma|_{\tU}$. 
By Proposition \ref{lem:ext}
$\ts$ admits an extension to $\tX$ denoted still 
by $\ts$ such that 
$\sigma\circ\ts=s\circ\tilde\sigma$.
The remaining assertions are easy consequences 
of the construction and the latter equality. 
\eproof
\section{Ellipticity of cones}\label{sec:cones} 
We introduce below the so called ``curve-orbit'' property $(*)$, 
see Definition \ref{def:co}. 
Using this property we give a criterion 
of ellipticity of the affine cones 
over projective varieties with the vertex removed. 

Recall that blowing up an affine cone 
at its vertex yields a line bundle 
on the underlined projective variety. 
Removing the vertex of the cone results 
in removing the zero section of the latter line bundle. 
This results in a locally trivial fiber bundle 
whose general fiber is
the punctured affine line $\A^1_*=\A^1\setminus\{0\}$.
In subsection \ref{ss:cones} we deal more generally with
locally trivial $\A^1_*$-fibrations. 
In subsection \ref{ss:twists} we extend our results 
to twisted $\A^1_*$-fibrations 
which are not locally trivial. In subsection \ref{ss:examples} 
we give several examples of elliptic cones, 
including cones over flag varieties.

The following fact is stated in 
\cite[3.5B$''$]{Gro89}. For the reader's convenience 
we give an argument. 
\bprop\label{prop:Gro}
Let $\rho\colon Y\to X$ be a locally trivial fiber bundle. 
If $Y$ is elliptic then $X$ is elliptic too. 
\eprop
\bproof Let us show that $X$ is locally elliptic. 
Let $(\hE,\hp,\hs)$ be a dominating spray on $Y$. For a point $x\in X$ 
choose a neighborhood $U$ of $x$ in $X$ such that the restriction 
$Y|_U\to U$ is a trivial fiber bundle. Let $\xi\colon U\to Y$ be a section 
of $Y|_U\to U$, let $p\colon E_U:=\xi^*\hE\to U$ be the induced 
vector bundle over $U$ and 
$\varphi\colon E_U\stackrel{\cong}{\longrightarrow} \hE|_{\xi(U)}$ 
be the induced isomorphism. 
Letting $s=\rho\circ\hs\circ\varphi\colon E_U\to X$ we obtain 
a local spray $(E_U, p, s)$ on $U$ with values in $X$.  
Indeed, let $\hZ$ be the zero section of $\hp\colon\hE\to Y$ and $Z_U$ 
be the zero section of $p\colon E_U\to U$. Since $\hs|_\hZ=\hp|_\hZ$ 
we have $s|_{Z_U}=p|_{Z_U}$.

Let $y=\xi(x)\in \hU=\rho^{-1}(U)$. Let  $0_y$ 
be the origin of the vector space $\hE_y=\hp^{-1}(y)$. 
Since $(\hE,\hp,\hs)$ is dominating, 
$d\hs|_{T_{0_y}\hE_y}\colon T_{0_y}\hE_y\to T_yY$ is surjective.  
This yields a surjection 
\[ds=d\rho\circ d\hs\circ d\varphi|_{T_{0_x}E_x}\colon  
T_{0_x}E_x\stackrel{\cong}{\longrightarrow} 
T_{0_y}\hE_y \to T_yY\to T_xX.\] 
Thus, the spray $(E_U,p,s)$ is dominating at $x$. 
This shows that $X$ is locally elliptic, hence elliptic, 
see \cite[Theorem 0.1]{KZ23}.
\eproof  
The natural question arises whether the converse 
implication holds. 
Clearly, the product $X\times\A^1_*$  is not elliptic even if $X$ 
is elliptic. 
However, we indicate below some particular settings 
where the answer is affirmative. 
\subsection{Technical lemmas}\label{ss:lemmas} 
In what follows $X$ stands for a smooth algebraic variety, 
$\rho \colon F \to X$ 
stands for a line bundle on $X$ with zero section $Z_F$, 
and $Y=F\setminus Z_F$
(except for Proposition \ref{prop:Gro}). 
Thus, $\rho|_Y\colon Y\to X$ is a locally trivial fiber bundle with fiber 
$\A^1_*=\A^1\setminus\{0\}$. In fact, any locally trivial fiber bundle $ Y\to X$ 
with general fiber $\A^1_*$ arises in this way. 

In Lemmas \ref{spr.l1}--\ref{spr.new.c}
we work as well in the $\ZZ/2\ZZ$-equivariant setup 
in order to prepare 
tools for dealing with twisted $\A^1_*$-fibrations, 
see subsection \ref{ss:twists}. 
In the sequel $\mu_2\cong\ZZ/2\ZZ$ stands 
for the group of square roots of unity.
\blem\label{spr.l1} Let $p: E\to X$ be a vector bundle with zero section
$Z$ and $\tau_i\colon F_i\to E$, $i=1,2$
be two line bundles on $E$. Assume that there exists an isomorphism 
$\varphi_0\colon F_1|_Z\to F_2|_Z$.
Then $\varphi_0$ extends to an isomorphism $\varphi \colon F_1\to F_2$. 

Furthermore, suppose that $X,E$ and $F_i|_Z$ 
are equipped with $\mu_2$-actions
such that $p, \,\varphi_0$ and $\tau_i|_Z$  for $i=1,2$
are $\mu_2$-equivariant.
Then there are $\mu_2$-actions on $F_i$ such that  $\varphi$ 
and $\tau_i$ for $i=1,2$ are $\mu_2$-equivariant.
\elem
\bproof  
The pullback yields an isomorphism $\Pic(E)\cong p^*\Pic(X)$, 
see \cite[Theorem 5]{Mag75}. 
Hence, $p^*(F_i|_Z)\cong F_i$ for $i=1,2$. 
Now the lemma follows. 
\eproof 
Due to the following lemma a spray on $X$ admits a pullback to $Y$. 
\blem\label{spr.p1} Any spray $(E,p,s)$ on $X$ induces a spray 
$(\hE,\hp,\hs)$ on $Y=F\setminus Z_F$ such that
$\hE$ fits in commutative diagrams
\begin{equation}\label{eq:1}
\begin{array}{ccc} \hE &  \stackrel{{\hp}}{\longrightarrow} & Y\\
\, \, \, \, \downarrow^{\hat \rho}   & 
& \, \, \, \,\,\, \downarrow^{\rho|_Y}\\
E &  \stackrel{{p}}{\longrightarrow} & X 
\end{array} 
\quad\text{and}\quad 
\begin{array}{ccc} \hE &  \stackrel{{\hs}}{\longrightarrow} & Y\\
\, \, \, \, \downarrow^{\hat \rho}   & 
& \, \, \, \,\,\, \downarrow^{\rho|_Y}\\
E &  \stackrel{{s}}{\longrightarrow} & X 
\end{array} 
\end{equation}

Furthermore, suppose there are $\mu_2$-actions on $X,E$ and $F$
(and, therefore, on $Y$)  such that
$p$, $s$ and  $\rho : F\to X$ are $\mu_2$-equivariant. Then
there is a $\mu_2$-action on $\hE$ such that 
$\hp$ and $\hs$ are $\mu_2$-equivariant.
\elem
\bproof   
Consider the line bundles $F_1=p^*F\to E$ and $F_2=s^*F \to E$ 
induced from $\rho\colon F\to X$ via morphisms $p\colon E\to X$ 
and $s\colon E\to X$, respectively.  They fit in commutative diagrams
\begin{equation}\label{eq:2}
\begin{array}{ccc} F_1&   \stackrel{\hp}{\longrightarrow} & F\\
\,    \downarrow   & & \,   \downarrow{\rho}\\
E &  \stackrel{{p}}{\longrightarrow} & X 
\end{array} 
\quad\text{and}\quad 
\begin{array}{ccc} F_2 &   \stackrel{\hs}{\longrightarrow} & F\\
\, \downarrow   & &   \downarrow{\rho}\\
E &  \stackrel{{s}}{\longrightarrow}  & X 
\end{array} 
\end{equation}
Since  $s|_Z=p|_Z$ we have $F_1|_Z=F_2|_Z\cong F$ 
under a natural identification 
of $Z$ and $X$. By Lemma \ref{spr.l1} 
there is an isomorphism
$\varphi \colon F_1\stackrel{\cong}{\longrightarrow} F_2$ 
of the line bundles 
$\hat\rho_1\colon F_1\to E$ and $\hat\rho_2\colon F_2\to E$. 
 Letting $Z_i$ be the zero section of $F_i$ and 
$\hat Y_i=F_i\setminus Z_i$ 
we get two isomorphic fiber bundles $\hat Y_i\to E$, $i=1,2$ 
with general fiber 
$\A^1_*=\A^1\setminus\{0\}$. Letting $\hat E=\hat Y_1$ 
and composing 
the isomorphism $\varphi|_\hE \colon 
\hat E\stackrel{\cong}{\longrightarrow}\hat Y_2$ 
with morphisms $\hat Y_2\to Y$ 
and $\hat Y_2\to E$ from the second diagram in \eqref{eq:2}  
yields \eqref{eq:1}. 
The vector bundle $\hp\colon \hE\to Y$ 
in the first diagram in \eqref{eq:1}  is  
induced from $p\colon E\to X$ via 
the morphism $\rho|_Y\colon Y\to X$. 
Let $\hat Z=\hat\rho^{-1}(Z)$ 
be the zero section  of $\hp\colon \hE\to Y$.   
Since  $s|_Z=p|_Z$  we have $\hat s|_\hZ=\hat p|_\hZ$. So    
$(\hE,\hp,\hs)$ is a spray on $Y$.

For the second statement recall that the isomorphism
$\varphi \colon F_1\stackrel{\cong}{\longrightarrow} F_2$  
of Lemma \ref{spr.l1}
 is $\mu_2$-equivariant.
Hence, also the above isomorphism 
$\varphi|_{\hE}: \hE\stackrel{\cong}{\longrightarrow} \hY_2$ 
 is $\mu_2$-equivariant.
The induced morphisms
$\hp|_{F_1}\colon F_1\to F$ and 
$\hs \colon F_2\to F$ in \eqref{eq:2}  
are $\mu_2$-equivariant too since 
so are $p$, $s$ and  $\rho$  
by our assumption.
\eproof
The next lemma guarantees the existence 
of pushforward sprays on quotients by $\mu_2$-actions.
\blem\label{spr.new.c} 
Under the assumptions of Lemma \ref{spr.p1} 
suppose in addition that the 
$\mu_2$-action on $X$ is free.
Letting 
\[X'=X/\mu_2,\quad Y'=Y/\mu_2,\quad 
E'=E/\mu_2\quad\text{and}\quad \hE'=\hE/\mu_2\]
consider the induced morphisms 
\[\rho' \colon Y'\to X',\quad  
p' \colon E'\to X',\quad
 s'\colon E'\to X', \quad \hp' \colon\hE'\to Y'
 \quad\text{and}\quad \hs '\colon \hE'\to Y'.\]
Then  $\rho'\colon Y'\to X'$ is a smooth $\A^1_*$-fibration.
Suppose further that
$p' \colon E'\to Y'$ and $\hp' \colon \hE'\to Y'$
are vector bundles \footnote{A priori, these 
are vector bundles locally trivial in \'etale topology.}.
Then $(E',p',s')$ is a spray on $X'$ and 
$(\hE',\hp',\hs')$ is a spray on $Y'$ 
such that $\rho' \circ \hs'=s'\circ \hp'$.
\elem
The proof is immediate and we leave it to the reader. 

Next we show that a spray admits pullback to 
an unramified $\mu_2$-covering.
\blem\label{z2} Let $(E,p,s)$ be a spray on $X$ 
and $\tau : \tX \to X$
be an unramified Galois covering 
with Galois group $\mu_2$. 
Let also $\tp\colon \tE\to \tX$ be 
the induced vector bundle on $\tX$.
Then there are a $\mu_2$-action on $\tE$ 
and a spray $(\tE,\tp,\ts)$ on $\tX$
such that both $\tp$ and $\ts$ are 
$\mu_2$-equivariant and
$s\circ \tilde\tau=\tau\circ \ts$ where 
$\tilde\tau\colon \tE\to E=\tE/\mu_2$ 
is the quotient morphism.
\elem
\bproof  The $\mu_2$-action 
on $\tX$ generates a $\mu_2$-action on
$\tX \times E$ trivial on the second factor. 
Recall that 
\[\tE=\tX\times_XE=\{ (\tx, e) \in \tX \times E 
\,|\, \tau (\tx) =p(e)\}.\]
Clearly, the natural projection
 $\tp\colon \tE \to \tX$ is $\mu_2$-equivariant.   
 Let $s'=s\circ \tilde \tau : \tE \to X$ and
 \[E'=\tX\times_X\tE=\{ (\tx,\te)\in 
 \tX\times \tE\,|\, \tau (\tx)=s'(\te)\}.\] 
Since $\tau (\tx) =\tau (\tx')$ if and only 
if $\tx'=\lambda. \tx$ for some $\lambda \in \mu_2$,
while $s'$ is constant on the $\mu_2$-orbits in $\tE$  
we see that
$E'$ is an unramified double cover of $\tE$. 
 Letting $\tZ$ be the zero section of $\tE$ 
 consider an isomorphism 
 $\varphi =\tp|_{\tZ} \colon \tZ
 \stackrel{\simeq}{\longrightarrow} \tX$. 
Notice that the preimage $Z'$ of
$\tZ$ in $E'$ consists of two irreducible components 
\[Z'_1=\{ (\varphi (z),z)| \, z \in \tZ\}\quad\text{and}\quad
Z'_2=\{ (\lambda .\varphi (z),z)| \, z \in \tZ\}\] 
where $\lambda\in\mu_2$ is an element of order $2$.
Notice also that  $\tp \colon \tE \to \tX \simeq\tZ$ 
induces a surjective morphism $E'\to Z'$.
Hence, $E'$ consists of two irreducible components 
$E'_i$, $i=1,2$ where $E'_i$ contains $Z'_i$. 
Taking the component $E'_1$
which contains $Z'_1$ as the graph of $\ts$ 
we get a lift $\ts \colon \tE \to \tX$. 
A natural lift to $E'$ of the $\mu_2$-action interchanges 
$E_1'$ and $E_2'$.
By construction, $\tau \circ \ts =s\circ \tilde \tau$.
In particular, 
\[\tau \circ \ts (\lambda. \te)= s\circ \tilde \tau (\lambda . \te)
= s (e)\quad\text{where}\quad
e =\tilde \tau (\te)=\tilde\tau (\lambda . \te).\] 
Thus, $\tau \circ \ts (\lambda. \te)=s(e)=\tau\circ \ts (\te)$. That is,
$\ts (\lambda. \te)$ coincides with either 
$\lambda . \ts (\te)$ or $\ts (\te)$. Since for
$\te \in \tZ\simeq Z'_1$ we have the former,
by continuity
$\ts (\lambda. \te)=\lambda . \ts (\te)$ for each 
$\te \in \tE\simeq E'_1$. This concludes the proof.
\eproof

\subsection{Sprays on untwisted $\A^1_*$-fibrations}\label{ss:cones}
Let $\rho\colon Y\to X$ be a locally trivial $\A^1_*$-fibration. 
Our aim is to give a criterion as to when the total 
space $Y$ is elliptic. 
Proposition \ref{prop:Gro} contains a necessary condition 
for this, whereas 
Proposition \ref{spr.c1} and Corollary \ref{spr.t1} 
provide sufficient conditions. 
\bprop\label{spr.c1} Given a locally trivial $\A^1_*$-fibration 
$\rho\colon Y\to X$
over an elliptic variety $X$ let $(E,p,s)$ 
be a spray on $X$ and 
$(\hE,\hp,\hs)$ be the induces spray on $Y$, 
see Lemma \ref{spr.p1}. 
Given points $x\in X$ 
and $y \in \rho^{-1} (x)$ 
let  
$(\hE_i,\hp_i,\hs_i)$, $i=1,2$  be rank $1$ sprays on $Y$ 
such that the images 
$O_i=\rho (\hO_i)$ in $X$ of the $\hs_i$-orbits
$\hO_i$ of $y$ are smooth curves. Suppose that
\begin{itemize}  
\item[(i)]  $(E,p,s)$ is dominating at $x$, and
\item[(ii)]
 $T_y\hO_1\neq T_y\hO_2$, while $T_x O_1=T_x O_2$,  that is,
 $O_1$ and $O_2$ are tangent at $x$.
 \end{itemize} 
 Then the family of sprays  $\{(\hE,\hp,\hs),\, (\hE_1,\hp_1,\hs_1),\, 
(\hE_2,\hp_2,\hs_2)\}$ 
 is dominating at $y\in Y$.
\eprop
\bproof If $\hO_y\subset\hE$ is the $\hs$-orbit of $y$
then $\rho (\hO_y)$ coincides with the $s$-orbit $O_x$, 
see the second diagram in \eqref{eq:1}. 
By (i) $(E,p,s)$ is dominating at $x$, hence 
$d\rho (T_y \hO_y)=T_xX$ and so, 
$\codim_{T_y Y} T_y \hO_y=1$. 
On the other hand,  (ii) implies that
$T_y \hO_1+T_y \hO_2\subset T_y Y$
 contains a nonzero vector $v$ such that $d\rho(v)=0$, that is, 
 $v\not\in T_y \hO_y$. 
Thus, $T_yY={\rm span}(T_y \hO_y,T_y \hO_1,\,T_y \hO_2)$, 
as needed. 
\eproof
The following notion is inspired by Proposition \ref{spr.c1}.
\bdefi\label{def:co}
Let $X$ be a smooth algebraic variety, $x\in X$ and let 
$C_x\simeq\PP^1$ be 
a smooth rational curve on $X$ passing through $x$. 
We say that the pair $(x, C_x)$ verifies the \emph{two-orbit property} 
if 
\begin{itemize}
\item[$(*)$] $C_x$ is union of two orbits $O_1$ and 
$O_2$ of rank 1 sprays $(E_1,p_1,s_1)$ 
and $(E_2,p_2,s_2)$ on $X$, respectively, where 
 $O_i\cong\A^1$ passes through $x$
 and the restriction 
 of $(E_i,p_i,s_i)$ on  $O_i$ is dominating at $x$.
\end{itemize}
We say that $X$ \emph{verifies the curve-orbit property}
if through each point $x \in X$ passes a smooth rational curve 
$C_x$ on $X$ satisfying $(*)$. If $C_x$ can be chosen from 
a specific family $\cF$ 
of rational smooth projective curves on $X$ then we say that $X$ 
\emph{verifies the curve-orbit property with respect to} $\cF$.
\edefi
\begin{rem} Notice 
that any one-dimensional orbit $O_x$ of a $\G_{\mathrm a}$-spray 
is  isomorphic to $\A^1$ and the restriction 
of this spray to $O_x$ is dominating at $x$. 
In the sequel we mostly deal with $\G_{\mathrm a}$-sprays.
\end{rem}
\bcor\label{spr.t1} Let $X$ be a smooth
variety, $\rho : F\to X$ be a 
line bundle and $Y=F\setminus Z_F$. Suppose that
\begin{itemize}
\item[{\rm (i)}] $X$ is elliptic and verifies the curve-orbit 
property with respect to 
a covering family  of smooth projective  rational curves 
$\{C_x\}$ on $X$, and
\item[{\rm (ii)}] $\rho\colon F\to X$ restricts to a nontrivial line bundle 
on each member $C_x$ of the family.
\end{itemize}
Then $Y$ is elliptic.
\ecor
\bproof Fix a dominating spray $(E,p,s)$ on $X$. 
Let $(\hE_i,\hp_i,\hs_i)$, $i=1,2$ and $(\hE,\hp,\hs)$ 
be the sprays on $Y$
induced by the $(E_i,p_i,s_i)$ and $(E,p,s)$, respectively, 
see Lemma \ref{spr.p1}. 
By Proposition \ref{spr.c1} it suffices to verify conditions (i) and (ii) 
of this proposition at any point $y\in\rho^{-1}(x)\cap Y$. 
Indeed, then the latter triple
of sprays is dominating at $y$. The family of all such triples for $x\in X$ 
is dominating on $Y$, 
which proves that $Y$ is subelliptic, hence elliptic by 
\cite[Theorem 0.1]{KZ23}. 

Since  the spray $(E,p,s)$ is dominating, condition (i) 
of Proposition \ref{spr.c1} holds. 
To verify (ii) of Proposition \ref{spr.c1}
we have to show that for each $y \in \rho^{-1} (x)\cap Y$ the 
$\hs_i$-orbits $\hO_1$ 
and $\hO_2$ of $y$
are transversal at $y$. By assumption (i) the $s_i$-orbits 
$O_i$, $i=1,2$ 
provide a trivializing cover  for the line bundle 
$\rho|_{C_x}\colon F|_{C_x}\to C_x$. 
Since $\rho|_{\hO_i}\colon \hO_i\to O_i$ is an isomorphism, 
$\hO_i$ is a 
non-vanishing local section of $F|_{C_x}$ over $O_i$. 
Thus, $\hO_i$ is a constant section 
given by $v_i=1$
in appropriate coordinates 
$(u_i,v_i)$ in $F|_{O_i}\simeq\A^1\times\A^1_*$.

We may consider that $x=\{u_i=1\}$ in $O_i$, $i=1,2$. 
Let $z=u_1/u_2$ be an affine coordinate in 
$\omega=O_1\cap O_2\simeq\A^1_*$.  
Due to (ii) the line bundle $\rho|_{C_x}\colon F|_{C_x}\to C_x$ 
is nontrivial. So, its transition function on $\omega$ equals 
$v_2/v_1=z^k$  for some $k\neq 0$. 
Therefore, in coordinates $(u_1,v_1)$ in 
$Y|_{O_1}\simeq \A^1\times\A^1_*$ the curve 
$\hO_2$ is given by equation $v_1= u_1^{-k}$. Hence, $\hO_2$ 
is transversal to $\hO_1$ at the point $y=(1,1)\in F|_{O_1}$. 
\eproof
\subsection{Ellipticity of twisted $\A^1_*$-fibrations}\label{ss:twists}
Let $X$ be a smooth complete variety. 
In this subsection we provide an analog of Corollary \ref{spr.t1} 
for any, not necessarily locally trivial smooth fibration 
$\rho\colon Y\to X$ 
with all fibers 
isomorphic to $\A^1_*=\A^1\setminus\{0\}$. 
Such a fibration extends to a locally trivial fibration
$\bar\rho\colon\bY\to X$ with general fiber $\PP^1$ 
and
with a divisor $\tX=\bY\setminus Y$ being smooth and
either irreducible or a union of two disjoint sections $Z_0$ 
and $Z_\infty$ of $\bar\rho$. The $\A^1_*$-fibration 
$\rho\colon Y\to X$ 
is said to be \emph{twisted}
in the former case and \emph{untwisted} in the latter case.

Letting in the untwisted case
$F=\bY\setminus Z_\infty$ we get 
a locally trivial line bundle $\bar\rho|_F\colon F\to X$ 
with zero section $Z_0$ so that $Y=F\setminus Z_0$. 
This returns us to the setup of Corollary \ref{spr.t1}. 

To extend Corollary \ref{spr.t1}  
in the twisted setup, we deal with $\mu_2$-varieties 
and $\mu_2$-vector bundles.
Indeed, a twisted $\A^1_*$-fibration $\rho\colon Y\to X$ 
can be untwisted as follows.
Consider the unramified Galois covering 
$\tilde\rho=\bar\rho|_{\tX}\colon \tX\to X$ 
with Galois group $\mu_2$ so that $X=\tX/\mu_2$.
The induced
$\PP^1$-fiber bundle $\tilde\rho^*\bY\to \tX$ over $\tX$
admits two disjoint sections $Z_0$ and $Z_\infty$ 
where $Z_0$ is the tautological section. Letting 
$\tF=\tilde\rho^*\bY\setminus Z_\infty$ we get a line bundle
$\tF\to \tX$. 
Hence, the induced $\A^1_*$-fibration 
$\tY=\tilde\rho^*\bY\setminus (Z_0\cup Z_\infty)\to \tX$ 
is untwisted. Furthermore, $\tY$ carries a free $\mu_2$-action 
making the projection 
$\tY\to \tX$ equivariant and such that $Y=\tY/\mu_2$. 
We can  now deduce the following analog of Corollary 
\ref{spr.t1} for twisted $\A^1_*$-fibrations.
\bprop\label{spr.t1+} Let $X$ be a smooth 
variety and $\rho\colon Y\to X$ be a smooth 
$\A^1_*$-fibration. 
Suppose that 
\begin{itemize}
\item[{\rm (i)}] $X$ is elliptic and verifies 
the curve-orbit property with respect to 
a covering family  of smooth projective  
rational curves $\{C_x\}$ on $X$;
\item[{\rm (ii)}] for any member $C_x$ 
of the family the restriction 
$\rho|_{\rho^{-1}(C_x)}\colon \rho^{-1}(C_x)\to C_x$ 
is a nontrivial $\A^1_*$-fiber bundle. 
\footnote{Condition (ii) 
is equivalent to the following one:
\begin{itemize}
\item[{\rm (ii$'$)}]
$\tF\to \tX$ restricts to a nontrivial line bundle 
on each of the two components of the curve 
$\tC_x=\tilde\rho^{-1}(C_x)$ on $\tX$.
\end{itemize}}
\end{itemize}
Then $Y$ is elliptic.
\eprop
\begin{figure}[ht]
\begin{tikzpicture}
\node at (-3.2,4) {$\hE'$};
\draw[->](-3.2,3.7)--node[right=1pt]{$\hat\rho'$} (-3.2,2.8);
\draw[->](-3,3.9)--node[above=3pt]{$/\mu_2$} (-1.6,3.9);
\draw[->](-3.4,3.9)--node[above=3pt]{$\hp'$} (-4.8,2.8);
\node at (-3.2,2.5) {$\tE$};
\draw[->](-3.4,2.4)--node[below=3pt]{$\tp$} (-4.8,1.3);
\draw[->](-3,2.5)--node[above=1pt]{$/\mu_2$} (-1.6,2.5);
\node at (-5,2.5) {$\tY$};
\draw[thick][->] plot [smooth,tension=0.3] coordinates{(-4.8,2.5) (-2.1,1.9)(0.2,2.5)};
\draw[->](-5,2.2)--node[left=1pt]{$\tilde\rho$} (-5,1.3);
\node at (-5,1) {$\tX$};
\draw[->](-4.8,1)--node[above=3pt]{$/\mu_2$} (-0.1,1);

\node at (-1.4,3.9) {$\hE$};
\draw[->](-1.4,3.7)--node[left=1pt]{$\hat\rho$} (-1.4,2.8);
\draw[->](-1.2,3.8)--node[above=3pt]{$\hp$} (0.2,2.7);
\node at (-1.4,2.4) {$E$};
\draw[->](-1.2,2.4)--node[below=3pt]{$p$} (0.2,1.3);
\node at (0.4,2.5) {$Y$};
\draw[->](0.4,2.2)--node[right=3pt]{$\rho$} (0.4,1.3);
\node at (0.4,1) {$X$};
\end{tikzpicture}

\begin{tikzpicture}
\node at (4.2,4) {$\hE'$};
\draw[->](4.2,3.7)--node[right=1pt]{$\hat\rho'$} (4.2,2.8);
\draw[->](4.4,3.9)--node[above=3pt]{$/\mu_2$} (5.8,3.9);
\draw[->](4,3.9)--node[above=3pt]{$\hs'$} (2.6,2.8);
\node at (4.2,2.5) {$\tE$};
\draw[->](4,2.4)--node[below=3pt]{$\ts$} (2.6,1.3);
\draw[->](4.4,2.5)--node[above=1pt]{$/\mu_2$} (5.8,2.5);
\node at (2.4,2.5) {$\tY$};
\draw[thick][->] plot [smooth,tension=0.3] coordinates{(2.6,2.5) (5.1,1.9)(7.6,2.5)};
\draw[->](2.4,2.2)--node[left=1pt]{$\tilde\rho$} (2.4,1.3);
\node at (2.4,1) {$\tX$};
\draw[->](2.6,1)--node[above=3pt]{$/\mu_2$} (7.3,1);

\node at (6,3.9) {$\hE$};
\draw[->](6,3.7)--node[left=1pt]{$\hat\rho$} (6,2.8);
\draw[->](6.2,3.8)--node[above=3pt]{$\hs$} (7.6,2.7);
\node at (6,2.4) {$E$};
\draw[->](6.2,2.4)--node[below=3pt]{$s$} (7.6,1.3);
\node at (7.8,2.5) {$Y$};
\draw[->](7.8,2.2)--node[right=3pt]{$\rho$} (7.8,1.3);
\node at (7.8,1) {$X$};
\end{tikzpicture}
\caption{}
\label {fig1}
\end{figure}
\bproof
Since $X$ is elliptic there exists 
a dominating spray $(E,p,s)$ on $X$.
Consider the natural $\mu_2$-action on $\tX$.
Lemma \ref{z2} guarantees the existence
of a pullback spray $(\tE,\tp,\ts)$ on $\tX$ such that 
the induced morphisms $\tp$ and $\ts$ 
are $\mu_2$-equivariant
with respect to the induced $\mu_2$-action on $\tE$.
By Lemma \ref{spr.p1} the latter spray admits 
a pullback to a $\mu_2$-equivariant 
spray $(\hE',\hp',\hs')$ on $\tY$, see Figure 1.  

The spray $(\hE',\hp',\hs')$ verifies the assumptions 
of Lemma \ref{spr.new.c}, 
with notation changed accordingly.
By this lemma, the  $\mu_2$-equivariant 
spray $(\hE',\hp',\hs')$ on $\tY$ induces a pushforward spray
$(\hE,\hp,\hs)$ on $Y$. Since the spray $(E,p,s)$ 
on $X$ is dominating, 
the resulting spray $(\hE,\hp,\hs)$ on $Y$ 
is dominating in codimension 1,
that is, for each $y \in Y$ 
and $v\in T_xX$ 
with $\rho(y)=x$ the $s$-orbit of $y$ 
is tangent to a vector $w \in T_yY$
with $d\rho(w)=v$. The curve-orbit property of $X$ together with
condition (ii) above guarantees
the existence of a rank 1 spray $(E_0,p_0,s_0)$ on $Y$
such that the $s_0$-orbit of $y$ is tangent 
to a nonzero vector from $\Ker (d\rho)$, 
see the proof of Proposition \ref{spr.c1}.
Now the conclusion follows by 
an argument from the latter proof.
\eproof
\subsection{Examples}\label{ss:examples}
\bexa\label{spr.exa1} Let $B$ be a smooth complete variety 
and $\pi \colon X \to B$ 
be a ruling, that is, 
a locally trivial fiber 
bundle over $B$ with general fiber $\PP^1$. Suppose that $X$ 
is elliptic and let $F\to X$ 
be a line bundle whose restriction to
a fiber of $\pi$ is nontrivial. 
Then the assumptions (i) and (ii) of Corollary \ref{spr.t1} 
hold with the family $\cF=
\{C_x\}_{x\in X}$ where $C_x=\pi^{-1}(\pi(x))$. 
According to this corollary $Y=F\setminus Z_F$ is elliptic.
\eexa
\bexa\label{spr.exa1+} The criterion of Corollary \ref{spr.t1} 
can be applied to a generalized 
affine cone $\hY$ over a smooth projective variety $X$ defined by 
an ample polarization $D$ on $X$, see the definition in the Introduction.
Indeed, blowing $\hat Y$ up at the vertex $v\in\hat Y$ yields 
the line bundle $F=\cO_X(-D)\to X$ so that $Y\setminus
\{v\}=F\setminus Z_F$. For instance, if 
$D$ is a hyperplane section of $X\subset \PP^n$
then $F=\cO_X(-1)$. Since $D$ is ample the condition (ii) of 
Corollary \ref{spr.t1} holds for any curve $C$ on $X$. 
Thus, we arrive at the following criterion. 
\eexa
\bcor\label{spr.cor2} 
Let a smooth projective variety $X$  be elliptic and
verify the curve-orbit property $(*)$, 
see Definition \ref{def:co}. Let $\hat Y\to X$ 
be a generalized affine cone  over $X$ with smooth locus 
$Y=\hat Y\setminus\{v\}$. 
Then $Y$ is elliptic.
\ecor
\bexa\label{spr.exa2} 1. The curve-orbit property holds for 
$\PP^1=\SL(2,\kk)/B_+$, where 
$B^\pm$ is a Borel subgroup of $\SL(2,\kk)$ consisting of 
upper/lower triangular matrices. Indeed, the unipotent radical 
$U^\pm$ of  $B^\pm$ acts on $\PP^1$ with a unique fixed point 
$x^\pm$ where 
$x^+\neq x^-$. By homogeneity, for any $x\in \PP^1$ 
we can choose a new pair of opposite Borel subgroups $\tB^\pm$ 
so that $x\notin\{\tx^+,\tx^-\}$ and so, the orbits $\tU^\pm x$ 
are one-dimensional and cover $\PP^1$. 

2. Any generalized flag variety $G/P$ with a reductive group 
$G$ over $\kk$ verifies the curve-orbit property. Indeed, 
by homogeneity of $G/P$ 
and due to
 the previous example it suffices to find an $\SL(2,\kk)$- 
 or $\PGL(2,\kk)$-subgroup 
 $S$ of $G$ and a point $x\in G/P$ such that the orbit $C=Sx$ 
 is isomorphic to $\PP^1$. In fact, we are in the setup of 
 the previous example because the isotropy subgroup of 
 a point $x\in C$ in $S$ is a Borel subgroup $B_S=S\cap P$ 
 of $S$ and
 the fixed 
 point $x^-$ of the opposite Borel subgroup $B_S^-$ 
 is different from $x$. 
 
Fix a a Cartan subalgebra $\fh$ 
of $\lie(P)\subset\lie (G)$, let 
$\Delta\subset\fh^\vee$ be the associated root system of $G$, $ \Delta^\pm$
be the subset of positive resp. negative roots in $\Delta$ and 
$\Delta_P\subset\Delta$ be the root system of $P$. We assume that 
$\Delta_P^+=\Delta^+$.
Consider the root decomposition 
\begin{equation}\label{eq:Cartan} \lie(P)=\fh\oplus 
 \bigoplus_{\alpha_i\in \Delta^+} 
 \fg_{\alpha_i} \oplus \bigoplus_{\alpha_i\in 
 \Delta_P^-} \fg_{\alpha_i} 
 \end{equation}
 where $\fg_{\alpha}$ is the root subspace of $\lie(G)$ 
which corresponds to a root $\alpha\in\Delta$. 

 Pick a negative root $\alpha\in\Delta^-\setminus \Delta_P^-$ 
 and let $S$ be the subgroup 
of $G$ with Lie algebra 
\[\lie(S)=\fg_{\alpha}\oplus 
\fg_{-\alpha}\oplus \fh_{\alpha}\quad\text{where}\quad  
\fh_{\alpha}=[\fg_{\alpha}, \fg_{-\alpha}]\subset\fh.\]
Thus, $S$ is isomorphic to either $\SL(2,\kk)$ or ${\rm PGL}(2,\kk)$ 
and $S\cap P=B_S$ is the Borel subgroup of $S$ with Lie algebra 
\[\lie(B_S)=\fg_{-\alpha}\oplus \fh_{\alpha}\,.\]
Since $B_S\subset B_P$, the $B_S$-action on $G/P$ 
fixes the distinguished point $x_0=eP\in G/P$. 
Let $U^-$ be the one-parameter unipotent subgroup of 
$S$ with Lie algebra $\fg_{\alpha}$.
Assuming that $U^-$ also fixes $x_0$ we obtain 
$[\fg_{\alpha}, \lie(P)]\subset\lie(P)$; in particular, 
$\fg_{\alpha}=[\fg_{\alpha}, \fh_{\alpha}] \subset \lie(P)$. 
The latter contradicts \eqref{eq:Cartan}. 
Therefore, $x_0\in (G/P)^{B_S}\setminus  (G/P)^{S}$ 
and so, the $S$-orbit $C=Sx_0$
 is  isomorphic to $\PP^1$.
\eexa
Recall that a variety $X$ is said to be \emph{of class $\cA_0$} if $X$ is 
covered by open subsets isomorphic
to the affine space $\A^n$ where $n=\dim X$.  
Such a variety is subelliptic, see 
\cite[Proposition 6.4.5]{For17a}. 
Therefore, it is elliptic by \cite[Theorem 0.1]{KZ23}. 

A generalized flag variety $G/P$ 
where $G$ is a reductive algebraic group and $P\subset G$ 
is a parabolic subgroup belongs to class $\cA_0$, 
see \cite[Sec. 4(3)]{APS14}; cf. also \cite[3.5F]{Gro89}. 
Indeed, a maximal unipotent subgroup $U$ of $G$ acts on $G/P$ 
with an open orbit isomorphic to an affine space, 
see, e.g., \cite[Sec. 3.1, proof of Proposition 7]{Akh95}. 
Since $G$ acts on $G/P$ transitively, the assertion follows.
See also \cite[Sec. 4]{APS14} for further examples.
\bcor\label{spr.cor3} Let $X=G/P$ be a flag variety and
$\hY$ be a generalized affine cone 
 over $X$. Then the smooth locus $Y=\hY\setminus \{v\}$ 
 of $\hY$ is elliptic. 
\ecor
\bproof
By the previous discussion and Example \ref{spr.exa2} $X$ 
verifies the assumptions of Corollary \ref{spr.cor2}. 
So, the assertion follows due to this Corollary. 
\eproof
\brem\label{rem:AKZ} Let $\hY$ be the affine cone 
over a flag variety $G/P$ 
embedded in $\PP^n$ as a projectively normal subvariety. 
Then $\hY$ is normal and moreover, flexible, 
see \cite[Theorem 0.2]{AKZ12}. Hence, $Y=\bY\setminus\{v\}$ is elliptic.
\erem
\section{Cones over  blown-up varieties of class $\cA_0$}\label{sec:A0}
\subsection{The enhanced curve-orbit property} 
Blowing up a variety $X$ of class $\cA_0$ with center at a point results again 
in a variety of class $\cA_0$, see, e.g., \cite[Sect. 3.5.D$''$]{Gro89} or 
\cite[Proposition 6.4.6]{For17a}. In particular, any rational smooth 
projective surface is of class $\cA_0$, see \cite[Corollary 6.4.7]{For17a}. 
The next proposition  allows to apply Corollary \ref{spr.cor2} 
to varieties of class $\cA_0$ blown up at a point.
\bprop\label{con.p1} Let $X$ be a complete variety of class $\cA_0$ 
and of dimension $n\ge 3$. 
Suppose that $X$ verifies the following 
\emph{enhanced curve-orbit property}:
\begin{itemize}
\item[$(**)$]
for each  $x \in X$ and any finite subset $M$ in $X\setminus \{x\}$ 
one can find a curve 
$C_x \simeq \PP^1$ on $X$ which avoids $M$ 
and verifies the two-orbit property $(*)$ of Definition \ref{def:co} with pairs of $\G_a$-sprays 
$(E_i,p_i,s_i)$, $i=1,2$.
\end{itemize}
Let $\sigma \colon \tX \to X$ be the blowup of a point $x_0\in X$.
Then $\tX$ also verifies $(**)$. 
\eprop
\bproof  
Given a point $\tilde x \in \tX$ fix a finite subset 
$\tM\subset \tX\setminus \{\tilde x\}$. 
Assume first that $x:=\sigma (\tilde x) \ne x_0$.
Choose a curve $C_{x}\subset X$  and a pair of extended 
$\G_{\mathrm a}$-sprays $(E_i,p_i,s_i)$ on $X$
satisfying $(**)$ with $M=\{x_0\}\cup\sigma(\tM)\subset X$. 
Replace the $(E_i,p_i,s_i)$
by suitable replicas such that the $s_i$-orbits of $x_0$ 
become singletons, while the 
$s_i$-orbits $O_i$ of $x$ do not change. 
By Proposition  \ref{cr.l2} one
can lift these sprays to extended $\G_{\mathrm a}$-sprays 
$(\tE_i,\tp_i,\ts_i)$ 
on $\tX$, $i=1,2$
such that $\sigma(\tO_i)=O_i$ where $\tO_i$ is the $\ts_i$-orbit 
of the point $\tx$.  
It is easily seen that the curve 
$C_{\tilde x}=\sigma^{-1} (C_{x})$ and the lifted sprays 
$(\tE_i,\tp_i,\ts_i)$ 
satisfy $(**)$.

Suppose now that $\tx$ lies on the exceptional divisor 
$\cE=\sigma^{-1}(x_0)$. 
Choose a neighborhood $\Omega\simeq\A^{n}$ of 
$\tx$ in $\tX$. Identifying 
$\Omega$ with $\A^n=\Spec\kk[u_1, \ldots, u_n]$ 
we may suppose that 
$x$ is the origin of $\A^n$. Then $(u_1\colon\cdots\colon u_n)$ 
can serve 
as homogeneous 
coordinates in $\cE\simeq\PP^{n-1}$.
Notice that $\tX$ can be covered with $n$ open subsets 
$U_i\simeq \A^n$ such that
$U_i\cap\cE=\cE\setminus H_i\simeq A^{n-1}$ where $H_i$ 
is the coordinate hyperplane in 
$\cE=\PP^{n-1}$ given by $u_i=0$. 

Choose the affine coordinates 
$(u_1, \ldots, u_n)$ in $X=\A^n$ 
in such a way that $\tx=(u_1\colon \cdots\colon u_n)\in\PP^{n-1}$ 
has two nonzero coordinates, say $u_1$ and $u_2$, that is, 
$\tx\in U_1\cap U_2$. Since $n-1\ge 2$,
a general projective line $C_{\tilde x}\subset\cE=\PP^{n-1}$ 
through $\tx$ meets 
$H_1$ and $H_2$ in two distinct points 
and does not meet $\tM$. Furthermore, for $i=1,2$
the affine line $\bar O_i=C_{\tilde x}\setminus 
H_i\simeq\A^1$ in $U_i$ 
is an orbit of a 
$\G_{\mathrm a}$-action on $U_i\simeq\A^n$ and so, a $\bar s_i$-orbit 
of the associated extended $\G_{\mathrm a}$-spray 
$(\bar E_i, \bar p_i,\bar s_i)$ 
on $\tX$. One has $C_{\tilde x}=\bar O_1\cup \bar O_2$.  
Thus, $(**)$ holds for $\tX$. 
\eproof
\brem\label{rem:n=2} Assuming that  $\tM$ is empty
the same construction of two extended 
$\G_{\mathrm a}$-sprays on $\tX$ with 
$C_{\tilde x}=O_1\cup O_2$ satisfying the 
two-orbit property  
goes through for $n=2$, while this time $\cE\simeq\PP^1$. 
\erem
\subsection{Enhanced curve-orbit property 
for  flag varieties}
\bprop\label{prop:enhanced} 
A generalized flag variety $X=G/P$ of dimension $n\ge 2$ 
verifies the enhanced curve-orbit property $(**)$.
\eprop
\bproof
The stabilizer of the distinguished 
point $x_0=eP\in G/P$ 
coincides with the parabolic 
subgroup $P$.  The construction of 
Example \ref{spr.exa2} produces a subgroup $S$ of $G$ 
isomorphic either to $\SL(2,\kk)$ or to $\PGL(2,\kk)$, 
or in other terms, a subgroup locally isomorphic 
to $\SL(2,\kk)$,
along with 
a pair of opposite one-parameter unipotent subgroups 
 $U^\pm$ of $S$ and
an $S$-orbit $C=S/B^+\simeq\PP^1$ 
passing through $x_0$
 such that the orbits $U^\pm x_0$ are one-dimensional and cover $C$,
 where $B^+=S\cap P$ 
 is a Borel subgroup of $S$. 
 
The conjugation of $S$ by elements $a\in P$ 
yields a family $S_a=aSa^{-1}$
of $\SL(2,\kk)$- resp. $\PGL(2,\kk)$-subgroups of $G$ 
along with a family of their orbits $C_a=aC\cong\PP^1$ 
passing through $x_0$. We claim that
the family of $\PP^1$-curves 
$\cF=\{C_a\}_{a\in P}$ has no base point 
 different from $x_0$. 
 Assuming that $\cF$ contains at least two distinct curves
 the base point set 
 ${\rm Bs}(\cF)$ is finite. 
 Since the parabolic subgroup $P$ is connected
 each $x\in {\rm Bs}(\cF)$ 
 is fixed by $P$.  Since $P$ coincides with 
 its normalizer $N_G(P)$, 
 see \cite[Theorem X.29.4(c)]{Hum95}, $x_0=eP$ 
 is the unique fixed point of $P$. 
 Indeed, if $PgP=gP$ for some $g\in G$, 
 then $g\in N_G(P)=P$ and so, $gP=P$. 

Let us show that ${\rm card}(\cF)\ge 2$. 
In the notation of \eqref{eq:Cartan} we have 
\[\dim(G/P)=\card(\Delta^-\setminus \Delta_P^-)\ge 2.\] 
Therefore, $\Delta^-\setminus\Delta^-_P$ contains 
two distinct negative roots
$\alpha_i$, $i=1,2$. 
Let $U^+_i$ be the 
root subgroup of $P$ with 
$\lie(U^+_i)=\fg_{-\alpha_i}$, $S_i$ be the subgroup 
in $G$ with
\[\lie(S_i)=\lie(\fg_{\alpha_i}, \fg_{-\alpha_i})\cong \sl(2,\kk)\] 
and 
$B_i=S_i\cap P\supset U^+_i$ 
be a Borel subgroup of $S_i$, 
cf. Example \ref{spr.exa2}. The $S_i$-orbit of the 
distinguished point $x_0$ is 
a $\PP^1$-curve $C_i=S_i/B_i$ on $G/P$. 
Notice that $U_1^+$ and $U_2^+$ are two distinct subgroups 
of the unipotent radical $R_{\mathrm u}(P)$. Let $P^-$ 
be the opposite
parabolic subgroup of $G$ and $B^-_i$ 
be the opposite Borel subgroup of $S_i$
with unipotent radical $U^-_i$. Then $U_1^-$ and $U_2^-$ 
are two distinct subgroups 
of $R_{\mathrm u}(P^-)$. The $R_{\mathrm u}(P^-)$-orbit of $x_0$ 
is open in $G/P$
and the isotropy subgroup of $x_0$ in $R_{\mathrm u}(P^-)$ 
is trivial. Hence the orbits 
$U^-_ix_0$, $i=1,2$ are distinct and so are the curves 
$C_i=\overline{U^-_ix_0}$, $i=1,2$, as claimed. 
 
 Thus, for any finite subset $M=\{x_1,\ldots,x_m\}
 \subset G/P\setminus\{x_0\}$ and for every $i=1,\ldots,m$
one can choose a curve $C_{a_i}=a_iC$ from our family
which does not meet $x_i$. The set of elements $a_i\in P$ 
with this property is open and dense in $P$.
Therefore, for a general $a\in P$ the curve $C_a$ 
does not intersect $M$.
Since $G/P$ is homogeneous the choice of a point 
$x_0$ is irrelevant. 
It follows that $G/P$ verifies condition $(**)$, 
see Example \ref{spr.exa2}. 
\eproof

\subsection{The main result in dimensions $\ge 3$}
\bthm\label{con.t1}  Let $X$ be a smooth complete variety and
$Y$ be the smooth locus of a generalized affine cone over $X$, see 
the Introduction.
Then $Y$ is elliptic in each of the following cases:
\begin{enumerate} 
\item[{\rm (a)}] $n=\dim X\ge 3$ and $X$ 
is obtained from a variety $X_0$ 
of class $\cA_0$ verifying condition $(**)$ of Proposition \ref{con.p1} 
(say, $X_0=G/P$ is a flag variety of dimension $n\geq 3$, 
see Proposition \ref{prop:enhanced}) 
via a finite number of subsequent blowups; 
 \item[{\rm (b)}] $X$ is obtained  by blowing up a finite subset of 
 a Hirzebruch surface $\F_n$, $n\ge 0$ or of $\PP^2$ (say, $X$ 
 is a del Pezzo surface). 
\end{enumerate}
\ethm
\bproof  Notice that in (a) and (b) $X$ is of class $\cA_0$, hence elliptic.
Thus, it suffices to check
conditions (i) and (ii) of Corollary \ref{spr.t1}. In all cases condition (ii) 
holds because of Corollary \ref{spr.cor2}.
Condition (i) 
is provided by Propositions \ref{con.p1} and \ref{prop:enhanced}, 
Remark \ref{rem:n=2} and 
Example \ref{spr.exa1}.
\eproof
\brem\label{rem:rational-surfaces} The proof of 
Theorem \ref{con.t1}(a) does not work for $n=2$. 
Indeed, in this case the exceptional curve $\cE=\PP^1$ in the proof 
of Proposition \ref{con.p1} 
can contain points of $\tM$. However, its conclusion still holds for $n=2$; 
see Theorem \ref{con.t2} in the next section.
\erem
\section{Elliptic cones over rational surfaces}\label{sec:surf}
In this section we extend Theorem \ref{con.t1}
to any rational smooth projective  surface $X$. 
\bthm\label{con.t2}  Let $X$ be a rational 
smooth projective surface and
$Y=\hY\setminus\{v\}$ be the smooth locus of 
a generalized cone $\hY$ over $X$. 
Then $Y$ is elliptic.
\ethm
Due to Corollary \ref{spr.cor2}, Theorem \ref{con.t2} follows 
from the next proposition. 
\begin{prop}\label{gcon.p1}
$X$ verifies the curve-orbit property $(*)$.
\end{prop}
The proof of Proposition  \ref{gcon.p1} uses induction 
on the number of blowups when forming $X$. 
The following proposition provides the inductive step.
\bprop\label{cen.l2}
Let $X$ be a smooth projective surface and $\pi\colon X\to\PP^1$ 
be a fibration with general fiber $\PP^1$. 
Assume that for any fiber component $C$ of $\pi$ and any point $x\in C$ 
the following hold. 
\begin{enumerate}
\item[$(i)$] There exists an open subset $U_{x}$ in $X$ 
such that $X\setminus U_{x}$ has no isolated point and
$O_x:=U_x\cap C=C\setminus\{x\}\cong \A^1$; 
\item[$(ii)$] there exists a proper birational morphism 
$\nu_x\colon U_x\to\A^2$ such that
\begin{itemize}
\item 
$O_x$ is sent by $\nu_x$ isomorphically onto a coordinate axis 
$L$ in $\A^2$;
\item the exceptional divisor of $\nu_x$ consists 
of some fiber components of $\pi$;
\item no one of these components is contracted to a point of $L$;
\end{itemize}
\item[$(iii)$] there is a $\G_{\mathrm a}$-action $\lambda_x$ on $U_x$ 
such that $O_x$ is the $\lambda_x$-orbit of $x$;
\item[$(iv)$] for the extended $\G_{\mathrm a}$-spray $(E_x,p_x,s_x)$ 
on $X$ associated with $\lambda_x$, $O_x$ is the $s_x$-orbit 
of any point $x'\in O_x$. 
\end{enumerate}
Let $\sigma\colon \tX\to X$ be the blowup of a point $x_0\in X$ and 
$\tilde\pi=\pi\circ\sigma\colon
\tX\to \PP^1$ be the induced $\PP^1$-fibration. 
Then $(i)-(iv)$  hold with  $\pi\colon X\to\PP^1$ replaced by 
$\tilde\pi\colon \tX\to\PP^1$.
\eprop
\bproof
Let $\tC$ be a fiber component of $\tilde\pi$ on $\tX$ 
and $\tx\in \tC$. If $\tC$ is different from the exceptional curve 
of $\sigma$ then we let 
$C=\sigma(\tC)\subset X$ and $x=\sigma(\tx)\in C$; 
otherwise, we let $x=x_0$. 
Consider an open subset $U_{x}\subset X$ satisfying (i)-(iv). 

The following cases can occur:
\begin{itemize}
\item[(a)]  $\tC=\sigma^{-1}(x_0)$ is the exceptional curve 
of $\sigma$ and 
$x=x_0$; 
\item[(b)] $C$ is a fiber component  of $\pi$ and $x_0\notin C$;
\item[(c)] $C$ is a fiber component  of $\pi$ and $x_0\in C$.
\end{itemize}

In case (a) we choose an open neighborhood 
$U_0\cong\A^2$ of $x_0$ in $X$. 
Such a neighborhood does exist since $X$ belongs 
to class $\cA_0$. 
We also choose a coordinate system $(u,v)$ on 
$U_0$ such that $x_0$ is the origin 
and the axis $L=\{u=0\}$ goes in direction of $\tx\in \tC=\PP T_{x_0}U_0$. 
Then $\tilde U_0:=\sigma^{-1}(U_0)=\tU_1\cup \tU_2$ 
where the $\tU_i\simeq\A^2$ are
equipped with coordinates $(u_i, v_i)$ such that 
\begin{itemize}
\item[($\alpha$)] the restrictions  
$\sigma_0|_{\tU_i}\colon \tU_i\to U_0$ are given  by 
\[\sigma_0|_{\tU_1}\colon (u_1,v_1)
\mapsto (u_1v_1,v_1)\quad\text{and}\quad 
\sigma_0|_{\tU_2}\colon (u_2,v_2)\mapsto (u_2, u_2v_2);\] 
\item[($\beta$)] $\sigma_0$ sends the coordinate axis 
$L_1:=\{u_1=0\}$ in $\tU_1$ 
isomorphically onto $L$. 
\end{itemize}
Then $U_{\tilde x}:=\tU_1$ equipped with  the 
$\G_{\mathrm a}$-action 
$\lambda_{\tilde x}$ by translations 
in direction of $L_1$ satisfies (i)--(iv).

If (b) or (c) takes place then the preimage 
$U_{\tx}=\sigma^{-1}(U_{x})$ in $\tX$ satisfies
\[O_{\tx}:=U_{\tx}\cap \tC=\sigma^{-1}(C\setminus \{x\})=
\tC\setminus\{\tx\}.\] 
Clearly, $U_{\tx}$ also satisfies condition (i). 

Assume that case (b) takes place. 
It is easily seen that the composition 
$\nu_{\tx}=\nu_{x}\circ\sigma\colon U_{\tx}\to\A^2$ 
satisfies (ii).  Replacing $\lambda_{x}$ 
with an appropriate replica, if necessary, we may assume that 
$x_0\notin C$ is fixed by 
$\lambda_{x}$, while $O_{x}=C\cap U_{x}=C\setminus\{x\}$ 
is still an orbit of $\lambda_{x}$. 
Thus, $\lambda_{x}$ can be lifted to a $\G_{\mathrm a}$-action 
$\lambda_{\tx}$ on 
$U_{\tx}$ such that  $O_{\tx}$ is the $\lambda_{\tx}$-orbit of $\tx$ 
and so, (iii) holds for $\lambda_{\tx}$.
Letting
$(E_{\tx},p_{\tx},s_{\tx})$ be the extended 
$\G_{\mathrm a}$-spray on $\tX$ 
associated with 
$\lambda_{\tx}$ we obtain a data verifying (i)--(iv). 

Case (c) splits in two sub-cases: 
\begin{itemize}
\item[$(c')$] $x=x_0$;
\item[$(c'')$] $x\neq x_0$. 
\end{itemize}
In case (c$'$) we have $x_0=x\notin U_x$. Hence, letting 
$U_{\tx}=\sigma^{-1}(U_{x})$ the restriction
$\sigma|_{U_{\tx}}\colon U_{\tx}\to U_x$ is an isomorphism. 
Proceeding as before we come to the desired conclusion. 

In case ($c''$) we have $x_0\in C\setminus\{x\}=O_x$ and 
$\nu_x(x_0)\in \nu_x(O_x)=L$. We may suppose that 
$\nu_x(x_0)=\bar 0\in \A^2$ is the origin. 
By (ii), $\nu_x\colon U_x\to\A^2$
 is \'etale on $O_x\subset U_x$. 
Hence, $\nu_x$ is a blowup of $\A^2$ 
with center an ideal $I\subset \cO_{\A^2}(\A^2)$ supported on 
$A^2\setminus L$. 
Since $\nu_x$ sends $O_x$ isomorphically onto $L$,
the proper birational morphism 
$\nu_x\circ\sigma\colon U_{\tx}\to\A^2$ also 
sends $O_{\tx}$ isomorphically onto $L$. 
Notice that the blowups with disjoint centers commute. 
Hence,  $\nu_x\circ\sigma=\sigma_0\circ \nu_0$ 
where $\sigma_0\colon \tU_0\to\A^2$ is the blowup of $\A^2$ 
at the origin and $\nu_0$ is the blowup of $\tU_0$ 
with center 
$\tilde I=\sigma_0^*(I)\subset \cO_{\tU_0}(\tU_0)$.
These morphisms fit in a commutative diagram 
\[
\begin{array}{ccc} U_{\tx} &  \stackrel{{\sigma}}{\longrightarrow} & U_{x}\\
\, \, \, \, \downarrow^{\nu_0}   & & \, \, \, \,\,\, \downarrow^{\nu_x}\\
\tU_0 &  \stackrel{{\sigma_0}}{\longrightarrow} & \A^2
\end{array}
\] 

Choose affine coordinates $(u, v)$ 
on $\A^2$ such that $x_0=(0,0)$ is the origin and $L=\{u=0\}$. 
Then $\tU_0=\tU_1\cup \tU_2$ where $\tU_i\simeq\A^2$ is
equipped with coordinates $(u_i, v_i)$ so that  ($\alpha$) 
and ($\beta$) above hold. 
Morphism 
$\sigma_0$ sends a coordinate axis $L_1$ in $\tU_1$ 
isomorphically onto $L$ and $\nu_0$ sends
$O_{\tx}$  isomorphically onto $L_1$. 

Replace now $U_{\tx}$ by $U_{\tx, 1}:=\nu_0^{-1}(\tU_1)$ 
and $\nu$ by 
$\nu_1=\nu_0|_{U_{\tx, 1}}\colon U_{\tx, 1}\to \tU_1\simeq\A^2$.
We have $O_{\tx}=U_{\tx, 1}\cap \tC$ and 
$\nu_1|_{O_{\tx}}\colon O_{\tx}\to L_1$ 
is an isomorphism which sends $x_0\in O_{\tx}$ 
to the origin $\bar 0\in \tU_1=\A^2$. 
Furthermore, $\nu_1$ is \'etale on $O_{\tx}$ and so, 
$U_{\tx, 1}$ verifies (i) and (ii). 

Consider now the $\G_{\mathrm a}$-action 
$\lambda$ on $\tU_1=\A^2$ by translations 
in direction of $L_1$ so that $L_1$ is a $\lambda$-orbit. 
Replacing $\lambda$ by an appropriate replica 
we can achieve that $\lambda$ 
 acts freely on $L_1$ and leaves invariant the center 
 $\tilde I|_{U_{\tx, 1}}$ of the blowup $\nu_0|_{U_{\tx, 1}}$.
Then $\lambda$ admits a lift $\tilde\lambda$ 
to $U_{\tx, 1}$ such that $O_{\tx}$ is a $\tilde\lambda$-orbit. 
Therefore,  $U_{\tx, 1}$ satisfies (iii) and (iv).
\eproof
\bproof[Proof of Proposition \ref{gcon.p1}]  
Clearly, $\PP^2$ and the Hirzebruch surfaces $\F_n$ 
with $n\ge 0$
verify the curve-orbit property. 
 It is well known that any rational 
smooth projective surface $X$ non-isomorphic to 
$\PP^2$ results from a sequence of blowups starting 
with $\F_n$ 
for some $n\ge 0$.
Moreover, properties (i)--(iv) of Proposition  \ref{con.p1} hold for 
the $\PP^1$-fibration $\pi\colon \F_n\to\PP^1$. 
By this proposition and the induction on the number of blowups
these properties still hold for $X$. Let $x\in X$ 
lies on a fiber component 
$C$ of the induced $\PP^1$-fibration $\pi\colon X\to \PP^1$. 
Choose two distinct points $x_1,x_2\in C\setminus \{x\}$. 
Let $U_i$ and 
$(E_i,p_i,s_i)$ verify  (i)--(iv) of Proposition \ref{con.p1} 
with $x$ replaced 
by $x_i$, $i=1,2$. Then $O_i=C\setminus \{x_i\}$ 
is an $s_i$-orbit of $x$. 
The restriction of the extended $\G_{\mathrm a}$-spray $(E_i,p_i,s_i)$ 
to $O_i$ is dominating. 
Therefore, $X$ verifies the curve-orbit property. 
This proves Proposition \ref{gcon.p1}.
\eproof
\section{Appendix A: High transitivity of the 
$\End(X)$-action}\label{app:ht}
It is known that the automorphism group $\Aut(X)$ 
of a flexible variety $X$ 
acts highly transitively on the smooth locus 
$X_{\rm reg}$, that is, it acts $m$-transitively 
for any $m\ge 1$, 
see \cite[Theorem 0.1]{AFKKZ13}. 
A similar transitivity property holds 
for the monoid of endomorphisms $\End(X)$ provided $X$ 
is a smooth elliptic variety in the Gromov' sense.  
Abusing the language we say that $\End(X)$
 \emph{acts highly transitively on $X$} if for 
 any two isomorphic zero-dimensional subschemes 
 $Z_1$ and $Z_2$ of $X$ every isomorphism 
 $Z_1\stackrel{\simeq}{\longrightarrow} Z_2$ 
 extends to an endomorphism of $X$.
\begin{prop}\label{prop:ht} Let $X$ be 
a smooth quasiaffine variety. 
If $X$ is elliptic  
then $\End(X)$ acts highly transitively on $X$.
\end{prop}
The above proposition is an easy consequence of 
the following epimorphism theorem due to 
Kusakabe \cite[Theorem 1.2]{Kus22}, 
which generalizes an earlier result of Forstneri\v{c} 
\cite[Theorem 1.6]{For17b} 
valid for complete elliptic varieties. 
This phenomenon 
was expected in \cite[Sec. 3.4(C)]{Gro89}; see also \cite{Arz22} 
for analogs
 in the case of flexible varieties.
\begin{thm} 
\label{thm:ForKusa}
For any smooth (sub)elliptic variety $X$ of dimension $n$ 
there exists a morphism 
$\phi\colon\A^{n+1}\to X$ such that 
$\phi(\A^{n+1}\setminus\Sing(\phi))=X$ 
where 
\[\Sing(\phi)=\{x\in \A^{n+1} | \phi \,\,\,\text{is not smooth at}\,\,\, x\}\] 
is the singular locus of $\phi$. 
If $X$ is complete then the same is true with $\A^{n}$ instead of $\A^{n+1}$.
\end{thm}
Using the fact that any morphism of a zero-dimensional scheme $Z\to X$ 
can be lifted through a smooth morphism 
$\phi\colon\A^{n+1}\setminus\Sing(\phi)\to X$, 
Kusakabe derives the following interpolation result, 
see \cite[Corollary 1.5]{Kus22}; cf. \cite[Sec. 3.4(E)]{Gro89}. 
\begin{cor}\label{cor:Kusa} Let $X$ be a smooth (sub)elliptic variety, 
$Y$ be an affine variety and $Z$ 
be a zero-dimensional subscheme of $Y$. 
Then for any morphism $f \colon Z \to X$ there exists a 
morphism $\tilde f \colon Y\to X$ such that $\tilde f|_Z = f$.
\end{cor}
In fact, Corollary \ref{cor:Kusa}  remains valid with the same proof 
if one assumes $Y$ to be quasiaffine. 
Letting in Corollary \ref{cor:Kusa} 
$Y=X$ and $Z_1=Z$  
we can extend the given morphism $f\colon Z_1\to Z_2$ 
onto a zero-dimensional subscheme $Z_2$ of $X$
to an endomorphism $\tilde f \colon X\to X$. 
This yields Proposition \ref{prop:ht}.
Furthermore, taking in Corollary \ref{cor:Kusa} $Y=\A^1$ 
we obtain the following corollary,  cf. \cite[Sec. 3.4(B)]{Gro89} 
and \cite[p.~1661]{KZ00}.
\begin{cor}\label{cor:Kus2} Any smooth elliptic variety $X$ is 
\emph{$\A^1$-rich}, that is, through any 
$m$ distinct points of $X$ passes an 
$\A^1$-curve $\A^1\to X$. Moreover, given a finite collection 
of curve jets $j_1,\ldots,j_m$ on $X$ 
with not necessarily distinct centers, there exists an 
$\A^1$-curve $f\colon \A^1\to X$ interpolating these jets. 
\end{cor}
Notice that for a  flexible smooth affine variety $X$
one can find an orbit of a $\G_{\mathrm a}$-action on $X$ 
interpolating given curve jets, see \cite[Theorem 4.14]{AFKKZ13}.  
\section{Appendix B: Gromov's Extension Lemma}\label{app:el}
For the reader's convenience we recall the formulation 
of Proposition \ref{lem:ext}. 
\begin{prop}\label{lem:ext-Shulim} 
Let $D$ be a reduced effective divisor on $X$
and $(E,p,s)$  be a spray  
on $U=X\setminus {\rm supp}(D)$  with values in $X$ 
such that $p\colon E\to X$ is a trivial 
bundle of rank $r\ge 1$.
Then there exists a spray $(\tilde E,\tilde p,\tilde s)$ on $X$ 
whose restriction to $U$ is isomorphic to $(E,p,s)$ and 
such that 
$\tilde s|_{{\tilde p}^{-1}(X\setminus U)}=
\tilde p|_{{\tilde p}^{-1}(X\setminus U)}$ 
and for each $x \in U$ the $\ts$-orbit of 
$x$ coincides with the $s$-orbit of $x$. 
\end{prop}
\begin{proof} Extend $E$ to a trivial vector bundle on $X$; 
abusing notation, the latter is  again denoted by $p\colon E\to X$. 
Let $\xi$ be a canonical section of $\cO_X(D)\to X$ 
such that ${\rm div}(\xi)=D$. 

Consider the twisted vector bundle 
$p_n\colon E_n=E\otimes \cO_X(-nD)\to X$ where $n\in\NN$.
Since $\xi$ does not vanish on $U$ 
the homomorphism $\phi_n\colon E_n\to E$ 
induced by the tensor multiplication by $\xi^n$ 
restricts to an isomorphism on $U$ which sends 
$s_n:=s\circ \phi_n$ to $s$. 
We claim that $s_n$ extends to a morphism 
$E_n \to X$ for a sufficiently
large $n$. It is easily seen that in the latter case
$(\tE, \tp,\ts)=(E_n, p_n, s_n)$ is a desired spray on $X$. 

The argument is local and, thus, 
we restrict consideration to 
 an affine neighborhood  $\omega\subset X$ 
 of a given point $x_0\in\supp(D)$
 such that 
 \begin{itemize}
\item $D|_\omega=h^*(0)$ for some 
$h\in\cO_\omega(\omega)$ and 
\item $E_n|_\omega\cong_\omega \omega\times \A^r$ 
is a trivial vector bundle.
\end{itemize}
 Notice that over 
 $\omega^*=\omega\setminus \supp(D)$ the spray
 $(E_n, p_n, s_n)$ is the $h^n$-homothety of $(E,p,s)$. 
 The latter means that $s_n$ is given, after trivialization, 
 by 
 \[s_n\colon (x,v)\mapsto s(x,h^n(x)v)\qquad 
 \forall (x,v)\in \omega^*\times\A^r.\] 

Let  $\omega\hookrightarrow \A^m$ 
be a closed embedding.
Then $s$ yields a rational map 
\[s\colon\omega\times \A^{r}\dashrightarrow 
\A^m\quad\text{of the form}
\quad (x, v) \mapsto x+ \psi(x,v)\]
where $\psi$ is a rational vector function
which is regular 
in a neighborhood of $\omega^*\times\{0\}$ and vanishes on 
$\omega^*\times\{0\}$.
Every coordinate function  of $\psi$ 
can be written in the form 
\[\frac{q(v)}{h^kr(v)}\quad\text{where}\quad v=(t_1, \ldots, t_r) 
\in\A^r\quad\text{and}\quad
q,r\in\cO_{\omega}(\omega)[t_1,\ldots,t_r]\] 
are such that $q(0)=0$ and the zero loci of $r$ and $h$ 
have no common component in $\omega\times\A^r$.

Furthermore,
we may suppose
that the divisors $r^*(0)$ and $q^*(0)$ have no common 
irreducible component $T$ in $\omega\times\A^r$ passing through 
$(x_0,0)\in\omega\times\A^r$. 
Indeed,  let $T$ be such a component. It is well known that 
$\Pic (\omega\times\A^r)={\rm pr}_1^*(\Pic (\omega))$. Hence,
 the divisor $T-{\rm pr}_1^* (T_0)$ is principal 
 for some divisor $T_0$ on $\omega$.  
Shrinking $\omega$ one may suppose that $T_0$ is principal. 
Hence, $T$ is principal,
 i.e.,  $T=f^*(0)$ for some regular function
 $f$ on $\omega\times \A^{r}$. Thus, dividing $r$ and $q$ 
 by $f$ we can get rid of $T$.

This implies that $r$ does not vanish on $\omega^*\times\{0\}$ 
since otherwise $\psi$ is not regular on $\omega^*\times\{0\}$.
In particular, $r(0)=ah^l$ where $a$ is a 
non-vanishing function on $\omega$.

The $h^n$-homothety $s_n$ of $s$ yields a rational map 
\[s_n\colon\omega\times \A^{r}\dashrightarrow \omega
\subset \A^m,\quad (x, v) \mapsto x+ \psi_n(x,v),\] 
where $\psi$ is a rational vector function
which is regular 
in a neighborhood of $\omega^*\times\{0\}$ and vanishes on 
$\omega^*\times\{0\}$.
Every coordinate function of  $\psi_n\in $ is of the form 
\[\frac{q(h^nv)}{h^kr(h^nv)}.\]
Since $q(0)=0$ and $[r(h^nv)-r(v)]|_{v=0}=0$ we see that 
\[q(h^nv)=h^n\tilde q(v)\quad\text{and}\quad
r(h^nv)=r(0) +h^n\tilde r(v)= h^la+h^n\tilde r(v).\] 
Hence, for $n> k+l$
the function 
\[\frac{q(h^nv)}{h^kr(h^nv)}=
\frac{h^{n-k-l}\tilde q(v)}{a+h^{n-l}\tilde r(v)}\]
is regular in a neighborhood $\Omega$ of $\omega\times\{0\}$ 
and vanishes to order $n-k-l$ 
on $({\rm supp}(D)\times\A^r)\cap\Omega$. 

Consider now $s_{n+1}=s\circ\phi_{n+1}=s_{n}\circ\phi_1$ where, 
after trivialization,
 \[\phi_{1}|_{\omega\times \A^r }\colon \omega\times \A^r 
 \to \omega\times \A^r, 
 \quad (x,v) \mapsto (x,h(x)v).\]
 Letting $\widetilde\Omega=\phi_{1} ^{-1} (\Omega) 
 \subset \omega\times\A^r$ we see that 
 \[\phi_{1} (\widetilde\Omega) \subset \Omega\quad\text{and}\quad 
 (\supp(D)\cap\omega)\times\A^r\subset \widetilde\Omega.\]
 Hence, $s_{n+1}$ is regular on $\widetilde\Omega$ and so, 
 on $p_{n+1}^{-1}(\supp(D)\cap\omega)$.
This yields a spray $(E_{n+1}, p_{n+1},s_{n+1})$ on 
$\omega\cup U$ with values in $X$. 
Choosing a finite cover of $\supp(D)$ by affine open subsets $\omega_i$ 
and a sufficiently large $n$ we get the desired extension of $(E,p,s)$ to 
a spray $(\tE,\tp,\ts)$ on $X$. 
\end{proof}
\brem The same proof works also for the original Gromov's Localization Lemma, 
see \cite[3.5.B]{Gro89}. 
\erem

\medskip

\noindent {\bf Acknowledgments.} We are grateful to Ivan Arzhantsev 
for useful consultation about flag varieties. 
His suggestions allowed to simplify 
the proofs in this part; the original proofs 
were much longer.
%
%%%%%%%%%%%%%%%%%%%%%%%%%%%%%%%%%%%%%%%%%%%%%%%%%%%%%%%%%%%%%%%%%%%%%%%
%%%%%%%%%%%%%%%%%%%%%%%%%%%%%%%%%%%%%%%%%%%%%%%%%%%%%%%%%%%%%%%%%%%%%%%

\end{document}